\documentclass[11pt]{article}

\usepackage[margin=1.1in]{geometry}
\usepackage{amsmath,amssymb,amsthm}
\usepackage{tikz}
\usetikzlibrary{arrows.meta, positioning, calc}
\usepackage[colorlinks=true,linkcolor=blue,citecolor=blue,urlcolor=blue]{hyperref}
\hypersetup{
  pdftitle={Hecke algebra representations from the Katz--Long--Moody construction},
  pdfauthor={Haru Negami},
  pdfkeywords={braid group; Iwahori--Hecke algebra; Temperley--Lieb algebra},
  pdfsubject={20F36, 20C08}
}

\newtheorem{theorem}{Theorem}[section]
\newtheorem{proposition}[theorem]{Proposition}
\newtheorem{lemma}[theorem]{Lemma}
\newtheorem{corollary}[theorem]{Corollary}

\newtheorem{problem}[theorem]{Problem}

\theoremstyle{definition}
\newtheorem{definition}[theorem]{Definition}
\newtheorem{example}[theorem]{Example}

\theoremstyle{remark}
\newtheorem{remark}[theorem]{Remark}

\newcommand{\GL}{\mathrm{GL}}
\newcommand{\Ker}{\operatorname{Ker}}
\newcommand{\spec}{\operatorname{spec}}
\newcommand{\kk}{\boldsymbol{k}}
\newcommand{\rLM}{\rho^{\mathrm{LM}}}
\newcommand{\rLMl}{\rho^{\mathrm{LM}}_{\lambda}}
\newcommand{\rKLM}{\rho^{\mathrm{KLM}}_{\lambda}}
\newcommand{\rDR}{\rho^{\mathrm{DR}}_{\lambda}}
\newcommand{\Nil}{\mathcal{N}}

\title{Hecke algebra representations from the\\ Katz--Long--Moody construction}

\author{Haru NEGAMI\thanks{Chiba University,
1-33 Yayoi-cho, Inage-ku, Chiba-shi, Chiba, Japan.
E-mail: \texttt{negamiharu@gmail.com}}}

\date{}

\begin{document}
\maketitle

\begin{abstract}
The Katz--Long--Moody (KLM) construction produces, from a representation
$\rho$ of the semidirect product $F_n \rtimes B_n$ of a free group and
the braid group and a parameter $\lambda$, a new representation on a
canonical quotient of $V^{\oplus n}$; it unifies the Long--Moody
construction with the multiplicative middle convolution for
Knizhnik--Zamolodchikov (KZ)-type equations. Writing $g_i = \rho(x_i)$ and $s_i = \rho(\sigma_i)$, we
prove that the intermediate representation $\rLMl$ satisfies
$\det\bigl(z I - \rLMl(\sigma_i)\bigr)
= \det(z I - s_i)^{\,n-1}\det(z I + s_i g_i)$; in particular, the
spectrum of the braid generators is independent of $\lambda$ and of the
twist $g_{i+1} g_i^{-1}$. In the rank-one case the construction produces
the unreduced and reduced Burau representations (at generic and
resonant parameters, respectively), for $t \neq 1$ and $t^n \neq 1$;
here and throughout, the identification with the Burau representation
concerns the restriction to $B_n$. For arbitrary input we compute the action on
the subspace removed in the quotient, obtain the characteristic
polynomial of the KLM braid generators as an exact quotient of
characteristic polynomials, and give a projection criterion for
semisimplicity of the intermediate Long--Moody operator. Over an
algebraically closed field of characteristic zero, for scalar braid part
and semisimple $g$, we classify at every $\lambda \neq 1$ when the
output has a quadratic minimal polynomial
with two distinct roots: for $n \geq 3$, exactly when
$(g - I)(g - \mu I) = 0$ for some $\mu \notin \{1, -1\}$ occurring in
$\spec(g)$, independently of $\lambda$; for $n = 2$, one further
double-resonant family appears. Semisimplicity of $g$ is a genuine
hypothesis: we exhibit non-semisimple inputs, namely $g = I + \Nil$ with
$\Nil \neq 0$ and $\Nil^2 = 0$, whose KLM quotient is nonetheless a
Hecke module,
realizing the permutation representation of the symmetric group. The
Hecke modules of the classification factor through a Temperley--Lieb
algebra, and for general input we give an exact three-strand criterion
for Temperley--Lieb factorization.
\end{abstract}



\section{Introduction}\label{sec:intro}

Representations of the braid group $B_n$ in which the images of the standard
generators $\sigma_1, \dots, \sigma_{n-1}$ satisfy a quadratic relation
$(\rho(\sigma_i) - a)(\rho(\sigma_i) - b) = 0$, $a, b \in \kk^{\times}$,
occupy a special place in low-dimensional topology and mathematical physics:
they are precisely the representations that factor through the
Iwahori--Hecke algebra of type $A_{n-1}$, and through them one obtains the
HOMFLY-PT invariants of links \cite{Jones87}, as well as large families of
unitary braid representations relevant to topological quantum computation
\cite{NSSFD08, DRW16}. The prototype is the Burau representation
\cite{Burau36}, whose generators have the two eigenvalues $1$ and $-t$.

There is a substantial topological literature on constructing Hecke
algebra representations. Lawrence realized Hecke algebra modules by
twisted homology and monodromy of configuration spaces
\cite{Lawrence90, Lawrence93}, obtaining the representations attached to
two-row Young diagrams, and extended the construction to representations of the Iwahori--Hecke
algebra associated with arbitrary multi-row Young diagrams
\cite{Lawrence96}. Bigelow
subsequently gave a homological construction of Iwahori--Hecke modules
and formulated a corresponding realization problem for the irreducible
ones \cite{Bigelow04}. Most directly relevant to the present paper,
Bigelow and Tian developed several generalizations of the Long--Moody
construction \cite{BigelowTian08}: they extended it to subgroups of
$B_n$, recovering the Gassner representation of the pure braid group;
they introduced a Lawrence-type construction based on mixed braid
groups $B_{n,m}$ and proved that a quadratic Hecke relation satisfied by
the input representation of $B_m$ is inherited by the output
\cite[Theorem 5.4]{BigelowTian08}; they observed a universal quadratic
relation for the Long--Moody matrices over the noncommutative group ring
$\mathbb{Z}[F_n \rtimes B_n]$ \cite[Proposition 5.1]{BigelowTian08}; and
they gave an explicit reduced version of the construction. Thus
\cite{BigelowTian08} already provides methods for producing Hecke
algebra representations by generalized Long--Moody procedures.
Functorial and categorical extensions of the Long--Moody construction
have been developed by Souli\'e \cite{Soulie19}.

The problem considered here is different and complementary. Starting
from the semidirect-product input
$\rho \colon F_n \rtimes B_n \to \GL(V)$ defined by the Artin action
(recalled in Section~\ref{sec:prelim}), we ask for \emph{intrinsic}
conditions on the matrices $g_i := \rho(x_i)$ and $s_i := \rho(\sigma_i)$
under which the output of the Katz--Long--Moody construction is of Hecke
type. In particular, we do not assume that an auxiliary input
representation already satisfies a Hecke relation; our aim is to detect
the quadratic relation directly from the spectral and Jordan-theoretic
data of the given representation, and to track the effect of the
quotient defining the construction.

In \cite{HiroeNegami23}, K.~Hiroe and the author introduced the
Katz--Long--Moody (KLM) construction, which unifies the Long--Moody
construction of braid group representations \cite{Long94}
with Katz's middle convolution algorithm \cite{Katz96, DettweilerReiter00,
DettweilerReiter07}. Given a representation
$\rho \colon F_n \rtimes B_n \to \GL(V)$ of the semidirect product of the
free group $F_n$ by the braid group, and a parameter
$\lambda \in \kk^{\times}$, the construction first produces an intermediate
representation $\rLMl$ of $F_n \rtimes B_n$ on $V^{\oplus n}$ (the
\emph{twisted Long--Moody construction}) and then passes to the quotient by
the invariant subspace $K + L$ originating in the multiplicative middle
convolution for Knizhnik--Zamolodchikov (KZ)-type equations; the result
is denoted $\rKLM$.
In the companion paper \cite{Negami25} the author showed that, upon
restriction to the pure braid group, the KLM construction corresponds to
Haraoka's multiplicative middle convolution for KZ-type equations
\cite{Haraoka12, Haraoka20}, and that the construction preserves a nondegenerate invariant
Hermitian form, which may be indefinite.

It is therefore natural to ask when the KLM construction produces
representations of the Iwahori--Hecke algebra, that is, when the operators
$\rKLM(\sigma_i)$ satisfy a quadratic relation. Since the minimal polynomial
of $\rKLM(\sigma_i)$ is governed by its eigenvalues and Jordan structure,
the first step is to compute the spectrum of the braid generators under the
construction. This is our first main result.

\begin{theorem}[Theorem \ref{thm:charpoly}]\label{thm:A}
Let $\rho \colon F_n \rtimes B_n \to \GL(V)$ be a representation on an
$N$-dimensional $\kk$-vector space $V$, and write $g_i := \rho(x_i)$,
$s_i := \rho(\sigma_i)$. Then, for every $\lambda \in \kk^{\times}$ and
every $i = 1, \dots, n-1$,
\[
\det\bigl(z I_{Nn} - \rLMl(\sigma_i)\bigr)
= \det(z I_N - s_i)^{\,n-1}\,\det(z I_N + s_i g_i).
\]
\end{theorem}

The proof rests on an explicit block triangularization of the local
$2N \times 2N$ block (see \eqref{eq:triangular} below), which in turn
uses only the semidirect product relation
$g_{i+1} = s_i g_i s_i^{-1}$. Theorem \ref{thm:A} has two immediate
structural consequences. First, the characteristic polynomial does not
involve $\lambda$: the convolution parameter enters the construction only
through the images of the free group generators, and hence, on the quotient,
only through the subspace $K + L$, i.e., through \emph{which} eigenvalues
are removed, never through their values. Second, if one writes
$g_{i+1} = C_i g_i$, then $C_i$ is forced to be the commutator
$[s_i, g_i]$, and the spectrum is entirely insensitive to it:
$\spec(s_i g_{i+1}) = \spec(s_i C_i g_i) = \spec(s_i g_i)$. Consequently, the \emph{characteristic polynomial} of the braid
generators depends only on the pair
$\bigl(\spec(s_i), \spec(s_i g_i)\bigr)$. We emphasize that the Hecke
property is a condition on the \emph{minimal} polynomial and is decided
by this pair only in combination with two further pieces of data: the
Jordan structure of the braid generators (Remark \ref{rem:jordan}) and
the action on the subspace $K + L$ removed in the quotient (Corollary
\ref{cor:quotient}). Since a Hecke relation $(T - a)(T - b) = 0$ with
$a \neq b$ forces the braid generators to act semisimply on the quotient,
restricting attention to semisimple braid images there entails no loss
of generality for the classification problem; under that standing hypothesis the eigenvalue
data become decisive. We stress that this concerns the braid images on
the quotient and not the input $g$: semisimplicity of the former does
not imply that of the latter, and non-semisimple $g$ can yield Hecke
modules (Remark \ref{rem:muone}). The hypothesis that $g$ be semisimple,
in force in Section \ref{sec:hecke}, is therefore a genuine restriction.

Our second main result is a complete classification in the case where the
braid part of the input representation is scalar, which is the natural
common generalization of the rank-one situation underlying the Burau
representation. Note that for any $g \in \GL(V)$ and $c \in \kk^{\times}$,
setting $\rho(x_j) = g$ for all $j$ and $\rho(\sigma_i) = cI_N$ defines a
representation of $F_n \rtimes B_n$ (Lemma \ref{lem:scalarrep}).

\begin{theorem}[Theorem \ref{thm:classification}]\label{thm:B}
Let $\kk$ be an algebraically closed field of characteristic zero, let
$g \in \GL(V)$ be semisimple, $c \in \kk^{\times}$, and let $\rho = \rho_{g,c}$
be the representation above. Let $\lambda \in \kk^{\times} \setminus \{1\}$.
\begin{enumerate}
\item[(i)] If $\spec(g) \subseteq \{1, \mu\}$ for some
$\mu \notin \{1, -1\}$, then the operators $\rKLM(\sigma_i)$ satisfy
\[
\bigl(T - c\bigr)\bigl(T + c\mu\bigr) = 0,
\]
and $\sigma_i \mapsto c^{-1}\rKLM(\sigma_i)$ defines a representation of
the Iwahori--Hecke algebra $H_n(\mu)$.
\item[(ii)] Conversely, suppose that for some
$\lambda \in \kk^{\times}\setminus\{1\}$ the KLM quotient is nonzero and
the operators $\rKLM(\sigma_i)$ have a quadratic minimal polynomial with
two distinct roots. If $n \geq 3$, then
$\spec(g) \subseteq \{1, \mu\}$ for some $\mu \notin \{1, -1\}$; since
this condition does not involve $\lambda$, the conclusion of (i) then
holds for every $\lambda \in \kk^{\times}\setminus\{1\}$. If
$n = 2$, the same conclusion holds at nonresonant parameters, and
exactly one further family appears at resonance:
$\spec(g) \setminus \{1\} = \{\nu, -\nu\}$ with $\lambda = \nu^{-2}$
(Theorem \ref{thm:classification}(iv)).
\end{enumerate}
\end{theorem}

The three possible eigenvalues of $g$ play sharply different roles. The
eigenvalue $1$ contributes a pair of eigenvalues $\{c, -c\}$ to the
intermediate representation, and this pair is absorbed \emph{exactly} by the
subspace $K$ in the quotient. An eigenvalue $\mu \notin \{1,-1\}$
contributes the second Hecke eigenvalue $-c\mu$. The eigenvalue $-1$,
however, is genuinely pathological: it contributes the eigenvalue
$-c \cdot (-1) = c$, so that the eigenvalue \emph{set} of
$\rKLM(\sigma_i)$ may still equal $\{c, -c\mu\}$, while a nontrivial Jordan
block for the eigenvalue $c$ appears and the minimal polynomial acquires the
factor $(T - c)^2$ (Remark \ref{rem:minusone}). Thus the slogan ``exactly
two eigenvalues'' must be read as a statement about the minimal polynomial,
not about the eigenvalue set. For $n = 2$, resonant parameters produce
two further phenomena --- a scalar degeneration and a double-resonant
family --- described in Theorem \ref{thm:classification}(iv), (v). The distinction is also meaningful
algebraically: when the only nontrivial eigenvalue of $g$ is $-1$, the
relation becomes $(T - c)^2 = 0$, which after rescaling is precisely the
defining relation of the repeated-root specialization $H_n(-1)$; what
fails is semisimplicity and the two-distinct-root condition, not the
quadratic relation itself.

In the rank-one case $N = 1$, $g = t$, $c = 1$, Theorem \ref{thm:B}
recovers the Burau representation: for generic $\lambda$ one obtains the
$n$-dimensional unreduced Burau representation, and at the resonant
parameter $\lambda = t^{-n}$ (for $t^{n} \neq 1$) the subspace $L$
becomes one-dimensional and
the quotient is the reduced Burau representation
(Example \ref{ex:burau}). The resonance set here is exactly the exceptional
set $E$ governing the degeneration of the invariant Hermitian form in
\cite{Negami25}.

For general input, the classification is governed by the action on the
subspace $K + L$ removed in the quotient, and this action can be
computed. We show that $L$ is equivariantly isomorphic to
$\Ker(\lambda\, g_1 \cdots g_n - I)$ with the braid generators acting
through $s_i$ (Lemma \ref{lem:Laction}), that the action on $K$
decomposes into restrictions of $s_i$ and a twisted swap whose
characteristic polynomial is
$\det\bigl(z^2 I - (s_i g_i s_i)|_{\Ker(g_i - I)}\bigr)$ (Lemma
\ref{lem:Kaction}), and we
deduce a closed formula for the characteristic polynomial of
$\rKLM(\sigma_i)$ as an exact quotient of characteristic polynomials
(Theorem \ref{thm:generalcharpoly}). A projection criterion settles
semisimplicity at the intermediate level (Proposition
\ref{prop:projection}). Together with the containment criterion
$(\rLMl(\sigma_i) - a)(\rLMl(\sigma_i) - b)V^{\oplus n} \subseteq K + L$,
which is equivalent to the Hecke relation on the quotient (Remark
\ref{rem:exact}), this reduces the Hecke question for arbitrary input
to explicit linear algebra.

The paper is organized as follows. Section \ref{sec:prelim} recalls the
twisted Long--Moody and KLM constructions. Section \ref{sec:charpoly}
proves Theorem \ref{thm:A} and derives its corollaries. Section
\ref{sec:hecke} proves Theorem \ref{thm:B}, discusses the Burau example
and the repeated-root specialization at $-1$, and establishes the
Temperley--Lieb factorization of the resulting Hecke modules (Remark
\ref{rem:TL}). Section \ref{sec:general}
treats the general (non-scalar) case, including an exact
Temperley--Lieb criterion (Remark \ref{rem:TLgeneral}). Section \ref{sec:discussion}
collects remarks on iteration, unitarity, and related open problems.
Figure \ref{fig:outline} summarizes the construction and the role of
each eigenvalue of $g$ in the scalar case.

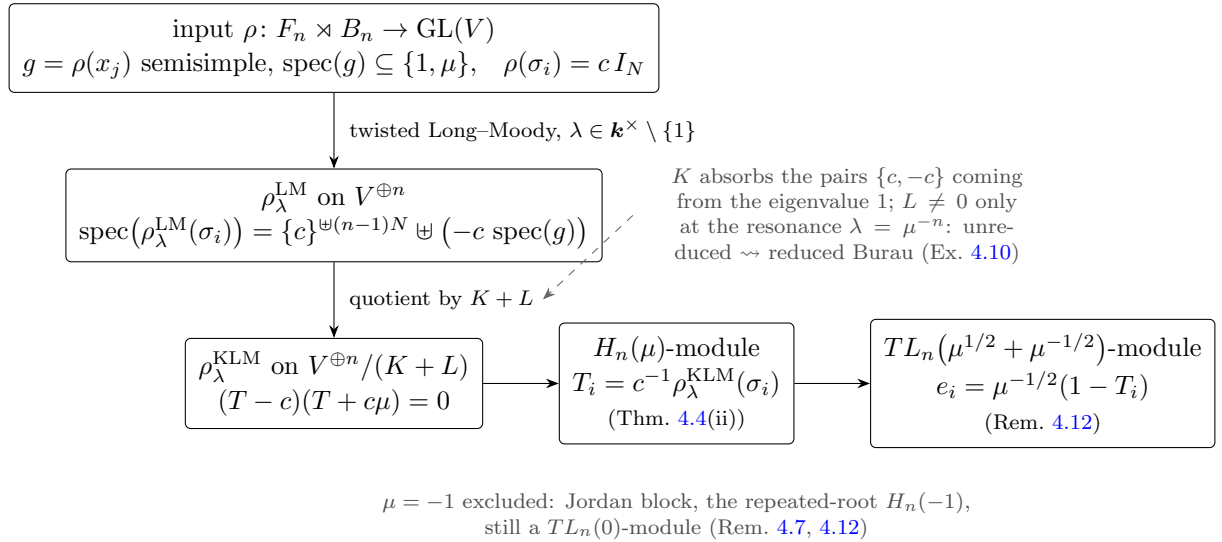
\begin{figure}[htbp]
\centering
\begin{tikzpicture}[
  >=Stealth,
  box/.style  = {draw, rounded corners=2pt, align=center,
                 inner sep=5pt, font=\small},
  note/.style = {align=center, font=\scriptsize, text=black!70},
  lab/.style  = {font=\scriptsize, align=center},
  node distance = 10mm and 10mm
]
\node[box] (rho)
  {input $\rho \colon F_n \rtimes B_n \to \GL(V)$\\[1pt]
   $g = \rho(x_j)$ semisimple, $\spec(g) \subseteq \{1, \mu\}$,\quad
   $\rho(\sigma_i) = c\, I_N$};

\node[box, below=of rho] (lm)
  {$\rLMl$ on $V^{\oplus n}$\\[1pt]
   $\spec\bigl(\rLMl(\sigma_i)\bigr)
    = \{c\}^{\uplus(n-1)N} \uplus \bigl(-c\,\spec(g)\bigr)$};

\node[box, below=of lm] (klm)
  {$\rKLM$ on $V^{\oplus n}/(K + L)$\\[1pt]
   $(T - c)(T + c\mu) = 0$};

\node[box, right=of klm] (hecke)
  {$H_n(\mu)$-module\\[1pt]
   $T_i = c^{-1}\rKLM(\sigma_i)$\\[1pt]
   {\scriptsize (Thm.~\ref{thm:classification}(ii))}};

\node[box, right=of hecke] (tl)
  {$TL_n\bigl(\mu^{1/2} + \mu^{-1/2}\bigr)$-module\\[1pt]
   $e_i = \mu^{-1/2}(1 - T_i)$\\[1pt]
   {\scriptsize (Rem.~\ref{rem:TL})}};

\draw[->] (rho)  -- node[lab, right=2pt]
  {twisted Long--Moody, $\lambda \in \kk^{\times}\setminus\{1\}$}   (lm);
\draw[->] (lm)   -- node[lab, right=2pt]
  {quotient by $K + L$} (klm);
\draw[->] (klm)  -- (hecke);
\draw[->] (hecke) -- (tl);

\node[note, right=4mm of lm, anchor=west, text width=54mm] (absorb)
  {$K$ absorbs the pairs $\{c, -c\}$ coming from the eigenvalue $1$; $L \neq 0$ only at the resonance $\lambda = \mu^{-n}$: unreduced $\rightsquigarrow$ reduced Burau (Ex.~\ref{ex:burau})};
\draw[->, dashed, black!60] (absorb.west) -- ($(lm.east)!.5!(klm.east)$);

\node[note, below=5mm of hecke]
  {$\mu = -1$ excluded: Jordan block, the repeated-root $H_n(-1)$,\\
   still a $TL_n(0)$-module (Rem.~\ref{rem:minusone}, \ref{rem:TL})};
\end{tikzpicture}
\caption{The construction of Temperley--Lieb representations from the
KLM construction with scalar braid part, in the two-distinct-root
regime of Theorem \ref{thm:classification}(ii). Within this family and
for $n \geq 3$, the convolution parameter $\lambda$ affects only the
dimension, through $L$, and the Hecke and Temperley--Lieb structure is
decided by $g$ alone; for $n = 2$ resonant parameters produce
additional phenomena.}
\label{fig:outline}
\end{figure}

\section{Preliminaries}\label{sec:prelim}

Throughout, $\kk$ denotes a field and $V$ a finite-dimensional
$\kk$-vector space, $N = \dim V$, and $n \geq 2$. We freely use the
notation of \cite{HiroeNegami23, Negami25} and recall only what is needed.

Let $B_n$ be the Artin braid group with generators
$\sigma_1, \dots, \sigma_{n-1}$ and let $F_n$ be the free group with
generators $x_1, \dots, x_n$. The Artin representation
$\theta \colon B_n \to \operatorname{Aut}(F_n)$ is defined by
\begin{equation}\label{eq:artin}
\theta_{\sigma_i}(x_j) =
\begin{cases}
x_{i+1} & j = i,\\
x_{i+1}^{-1} x_i x_{i+1} & j = i+1,\\
x_j & j \neq i, i+1,
\end{cases}
\end{equation}
and $F_n \rtimes B_n = F_n \rtimes_{\theta} B_n$ denotes the corresponding
semidirect product, in which $\sigma x \sigma^{-1} = \theta_{\sigma}(x)$
for $\sigma \in B_n$, $x \in F_n$.

Let $\rho \colon F_n \rtimes B_n \to \GL(V)$ be a representation and set
$g_i := \rho(x_i)$, $s_i := \rho(\sigma_i)$. Applying $\rho$ to the
relation $\sigma_i x_i \sigma_i^{-1} = x_{i+1}$ yields the identity
\begin{equation}\label{eq:key}
g_{i+1} = s_i\, g_i\, s_i^{-1}, \qquad \text{equivalently} \qquad
g_{i+1} s_i = s_i g_i,
\end{equation}
which will be used repeatedly.

\begin{definition}[twisted Long--Moody construction
\cite{HiroeNegami23, Negami25}]\label{def:LM}
Let $\lambda \in \kk^{\times}$. The twisted Long--Moody construction
associates to $\rho$ the representation
$\rLMl \colon F_n \rtimes B_n \to \GL(V^{\oplus n})$ defined on generators by
\[
\rLMl(x_i) := \rDR(x_i), \quad
\rLMl(\sigma_i) := s_i^{\oplus n} \cdot
\begin{pmatrix}
I_{N(i-1)} & & \\
& R_i & \\
& & I_{N(n-i-1)}
\end{pmatrix}, \quad
R_i :=
\begin{pmatrix}
O & g_i \\ I_N & I_N - g_{i+1}
\end{pmatrix},
\]
where $\rDR(x_i)$ is the Dettweiler--Reiter convolution matrix, namely the
identity matrix except in the $i$-th block row, which equals
\[
\bigl(\lambda(g_1 - 1), \dots, \lambda(g_{i-1} - 1),\ \lambda g_i,\
g_{i+1} - 1, \dots, g_n - 1\bigr).
\]
\end{definition}

Note that $\rLMl(\sigma_i) = \rLM(\sigma_i)$ does not involve $\lambda$;
the parameter enters only through the free group part.

\begin{definition}[Katz--Long--Moody construction
\cite{HiroeNegami23}]\label{def:KLM}
Define the subspaces $K, L \subseteq V^{\oplus n}$ by
\[
K := \bigl\{(w_1, \dots, w_n)^{T} ;\ w_j \in \Ker(g_j - 1)\ (1 \le j \le n)
\bigr\},
\qquad
L := \bigcap_{j=1}^{n} \Ker\bigl(\rLMl(x_j) - I_{Nn}\bigr).
\]
Then $K + L$ is $\rLMl$-invariant, and the Katz--Long--Moody construction
$\rKLM$ is the induced representation of $F_n \rtimes B_n$ on
$V^{\oplus n} / (K + L)$.
\end{definition}

We shall use the following characterization of $L$, due to
Dettweiler--Reiter \cite[Section~2.1]{DettweilerReiter07} (see also
\cite[Eq.~(9)]{Negami25}). For $1 \le i \le n$, let $E_i(v) \in V$
denote the $i$-th block of $\bigl(\rLMl(x_i) - I_{Nn}\bigr)v$; since
$\rDR(x_i)$ differs from the identity only in the $i$-th block row, one
has $L = \{v \,;\, E_i(v) = 0,\ 1 \le i \le n\}$. A direct computation
from Definition \ref{def:LM} gives
\begin{equation}\label{eq:Ediff}
E_{i+1}(v) - E_i(v) = (1 - \lambda)\bigl(v_i - g_{i+1} v_{i+1}\bigr),
\qquad 1 \le i \le n-1,
\end{equation}
so that for $\lambda \neq 1$ the $n$ conditions are equivalent to a
chain: $v = (v_1, \dots, v_n)^{T} \in V^{\oplus n}$ belongs to $L$ if and
only if
\begin{equation}\label{eq:Lchar}
v_k = g_{k+1} v_{k+1} \quad (1 \le k \le n-1), \qquad
v_n = \lambda\,(g_1 \cdots g_n)\, v_n.
\end{equation}
At $\lambda = 1$ the right-hand side of \eqref{eq:Ediff} vanishes
identically and \eqref{eq:Lchar} fails; the degenerate case is described
in Remark \ref{rem:lambdaone} below. Accordingly, every use of
\eqref{eq:Lchar} in what follows carries the hypothesis
$\lambda \neq 1$.

Finally, we fix our convention for the Hecke algebra. For
$q \in \kk^{\times}$, the Iwahori--Hecke algebra $H_n(q)$ of type
$A_{n-1}$ is the $\kk$-algebra with generators $T_1, \dots, T_{n-1}$,
the braid relations, and the quadratic relations
\begin{equation}\label{eq:heckerel}
(T_i - 1)(T_i + q) = 0, \qquad i = 1, \dots, n-1.
\end{equation}
A representation of $B_n$ in which the images $A_i$ of $\sigma_i$ satisfy
$(A_i - a)(A_i - b) = 0$ with $a \in \kk^{\times}$ therefore factors
through $H_n(-b/a)$ after rescaling $\sigma_i \mapsto a^{-1} A_i$.
When $-b/a = -1$, the defining polynomial has a repeated root; thus
factorization through $H_n(-1)$ must be distinguished from the condition
that the braid generators have two distinct eigenvalues.

\section{The characteristic polynomial of the braid
generators}\label{sec:charpoly}

\begin{theorem}\label{thm:charpoly}
Let $\rho \colon F_n \rtimes B_n \to \GL(V)$ be a representation,
$g_i = \rho(x_i)$, $s_i = \rho(\sigma_i)$, and let
$\lambda \in \kk^{\times}$. Then, for $i = 1, \dots, n-1$,
\begin{equation}\label{eq:charpoly}
\det\bigl(z I_{Nn} - \rLMl(\sigma_i)\bigr)
= \det(z I_N - s_i)^{\,n-1}\,\det(z I_N + s_i g_i).
\end{equation}
In particular, the characteristic polynomial of $\rLMl(\sigma_i)$ is
independent of $\lambda$, and as a multiset
\begin{equation}\label{eq:multiset}
\spec\bigl(\rLMl(\sigma_i)\bigr)
= \spec(s_i)^{\uplus (n-1)} \uplus
\bigl(-\spec(s_i g_i)\bigr),
\end{equation}
where spectra are taken in an algebraic closure of $\kk$ and counted with
multiplicity; $\uplus$ denotes the union of multisets, adding
multiplicities, and $M^{\uplus k}$ the $k$-fold multiset union of $M$
with itself.
\end{theorem}

\begin{proof}
By Definition \ref{def:LM}, $\rLMl(\sigma_i)$ is block diagonal with
$(k,k)$-block $s_i$ for $k \neq i, i+1$ and with middle
$2N \times 2N$ block
\[
B :=
\begin{pmatrix}
O & s_i g_i \\ s_i & s_i(I - g_{i+1})
\end{pmatrix}
\]
in the block rows and columns $i, i+1$. Hence
\begin{equation}\label{eq:split}
\det\bigl(z I_{Nn} - \rLMl(\sigma_i)\bigr)
= \det(z I_N - s_i)^{\,n-2}\,\det(z I_{2N} - B).
\end{equation}
Set
\[
P_i :=
\begin{pmatrix}
g_{i+1} & I \\ I & O
\end{pmatrix},
\qquad
P_i^{-1} =
\begin{pmatrix}
O & I \\ I & -g_{i+1}
\end{pmatrix}.
\]
A direct computation gives
\[
B P_i =
\begin{pmatrix}
s_i g_i & O \\ s_i g_{i+1} + s_i(I - g_{i+1}) & s_i
\end{pmatrix}
=
\begin{pmatrix}
s_i g_i & O \\ s_i & s_i
\end{pmatrix},
\]
and hence
\begin{equation}\label{eq:triangular}
P_i^{-1} B P_i =
\begin{pmatrix}
s_i & s_i \\ s_i g_i - g_{i+1} s_i & -g_{i+1} s_i
\end{pmatrix}
=
\begin{pmatrix}
s_i & s_i \\ O & -s_i g_i
\end{pmatrix},
\end{equation}
where the last equality uses \eqref{eq:key} in the form
$g_{i+1} s_i = s_i g_i$. The block upper triangular form
\eqref{eq:triangular} gives
\[
\det(z I_{2N} - B) = \det(z I - s_i)\,\det(z I + s_i g_i),
\]
which combined with \eqref{eq:split} yields \eqref{eq:charpoly}. We also
record, for later use, that conjugating by $s_i$ and using
\eqref{eq:key} once more gives
$s_i(z I + s_i g_i)s_i^{-1} = z I + s_i g_{i+1}$, so that
$\spec(s_i g_{i+1}) = \spec(s_i g_i)$.
\end{proof}

\begin{remark}[relation to the universal quadratic of
Bigelow--Tian]\label{rem:universal}
Bigelow and Tian observed that the \emph{universal} Long--Moody matrices
$\phi(\sigma_i)$, with entries in the noncommutative group ring
$\mathbb{Z}[F_n \rtimes B_n]$, satisfy a quadratic relation
$(\phi(\sigma_i) + \sigma_i x_i)(\phi(\sigma_i) - \sigma_i) = 0$
\cite[Proposition 5.1]{BigelowTian08}, and noted that it does not yield a
Hecke relation after evaluation, since its ``scalars'' do not commute
with the matrix entries. Theorem \ref{thm:charpoly} identifies exactly
what survives evaluation by a representation $\rho$: the noncommutative
quadratic descends to the factorization \eqref{eq:charpoly} of the
characteristic polynomial, with the two factors contributing the spectra
$\spec(s_i)$ and $-\spec(s_i g_i)$, while the triangular form
\eqref{eq:triangular} is its matrix-level counterpart, in which the
off-diagonal coupling block records precisely the failure of the
scalars to commute.
\end{remark}

\begin{remark}\label{rem:onlykey}
The only representation-theoretic input in the proof is the single relation
\eqref{eq:key}, i.e., $\theta_{\sigma_i}(x_i) = x_{i+1}$, together with
the specific form of the local block $B$. In particular, the
triangularization \eqref{eq:triangular} persists for any
Long--Moody-type construction whose evaluated local block has this form
and whose input satisfies $g_{i+1} s_i = s_i g_i$; for other Wada-type
actions \cite{Wada92} the Fox Jacobian, and hence the local block
itself, may differ. Of course, \eqref{eq:key} is only
one of the relations imposed on a representation of $F_n \rtimes B_n$ by
the Artin action \eqref{eq:artin}; throughout, $\rho$ is assumed to be a
representation of the full semidirect product.
\end{remark}

\begin{remark}[independence of the twist]\label{rem:twist}
Write $g_{i+1} = C_i g_i$. By \eqref{eq:key} the ``twist'' $C_i$ is not a
free parameter: it is the commutator
\[
C_i = g_{i+1} g_i^{-1} = s_i g_i s_i^{-1} g_i^{-1} = [s_i, g_i]
\qquad \text{(with the convention } [a, b] := a b a^{-1} b^{-1}\text{)}.
\]
Moreover, $\spec(s_i g_{i+1}) = \spec(s_i C_i g_i) = \spec(s_i g_i)$ by
the conjugation identity in the last step of the proof above.
Consequently, the spectrum of the braid generators under the (twisted)
Long--Moody construction depends only on the pair
$\bigl(\spec(s_i), \spec(s_i g_i)\bigr)$ and cannot be altered by any
choice of $C_i$. We stress that this independence concerns the
characteristic polynomial of the \emph{intermediate} representation
$\rLMl$. The subspaces $K$ and $L$, and hence the KLM quotient, depend
on the matrices $g_1, \dots, g_n$ themselves and not only on the spectra
above: for instance, for $\lambda \neq 1$ one has, by \eqref{eq:Lchar},
$\dim L = \dim \Ker(\lambda\, g_1 \cdots g_n - I_N)$, which involves the
full product $g_1 \cdots g_n$ and therefore the twists $C_j$. Thus the
$C_i$-independence of the local characteristic polynomial does not
extend verbatim to the quotient.
\end{remark}

\begin{remark}[Jordan structure]\label{rem:jordan}
The conjugated form \eqref{eq:triangular} also controls the Jordan
structure of $\rLMl(\sigma_i)$: the operator is similar to a block upper
triangular matrix with diagonal blocks $s_i$ and $-s_i g_i$, coupled by
the off-diagonal block $s_i$. If $s_i$ and $s_i g_i$ are semisimple and
$\spec(s_i) \cap \bigl(-\spec(s_i g_i)\bigr) = \emptyset$, the coupling
is removed by solving a Sylvester equation and $\rLMl(\sigma_i)$ is
semisimple. When the two spectra overlap, nontrivial Jordan blocks can
occur even for semisimple input (Remark \ref{rem:minusone}). Hence the
characteristic polynomial \eqref{eq:charpoly} alone does not decide the
Hecke property; it does under the additional hypothesis that the braid
generators act semisimply on the quotient---a hypothesis which is
automatic whenever a Hecke relation with two distinct parameters is the
target, since such a relation has squarefree minimal polynomial.
\end{remark}

We record the effect of passing to the KLM quotient. If
$U \subseteq V^{\oplus n}$ is an invariant subspace of an operator $A$,
then $\det(z - A) = \det(z - A|_{U})\cdot\det(z - \bar A)$, where
$\bar A$ is the induced operator on $V^{\oplus n}/U$. Applying this to
$U = K + L$ gives:

\begin{corollary}\label{cor:quotient}
As multisets,
\[
\spec\bigl(\rKLM(\sigma_i)\bigr)
= \Bigl[\spec(s_i)^{\uplus(n-1)} \uplus \bigl(-\spec(s_i g_i)\bigr)\Bigr]
\ \setminus\ \spec\Bigl(\rLMl(\sigma_i)\big|_{K+L}\Bigr).
\]
In particular, if $\rKLM(\sigma_i)$ has at most two distinct eigenvalues
$\{a, b\}$, then
\[
\spec(s_i) \cup \bigl(-\spec(s_i g_i)\bigr) \subseteq
\{a, b\} \cup \spec\Bigl(\rLMl(\sigma_i)\big|_{K+L}\Bigr).
\]
\end{corollary}

Thus the convolution parameter $\lambda$, which enters only through
$K + L$, can remove eigenvalues from the list \eqref{eq:multiset} but can
never create new ones. For $n \geq 3$ this makes the two-distinct-root
Hecke property of the scalar family independent of $\lambda$ (Theorem
\ref{thm:classification}(iii)); for $n = 2$, resonance can remove an
obstruction and thereby create it (Theorem
\ref{thm:classification}(iv)).

\begin{remark}[exact criterion on the quotient]\label{rem:exact}
Corollary \ref{cor:quotient} accounts for eigenvalues with algebraic
multiplicity; the Jordan structure is handled exactly by the following
observation. For any polynomial $p$, the induced operator satisfies
$p\bigl(\rKLM(\sigma_i)\bigr) = 0$ if and only if
\[
p\bigl(\rLMl(\sigma_i)\bigr)\, V^{\oplus n} \;\subseteq\; K + L .
\]
In particular, the two-distinct-root Hecke condition on the quotient is
equivalent to the containment
$\bigl(\rLMl(\sigma_i) - a\bigr)\bigl(\rLMl(\sigma_i) - b\bigr)
V^{\oplus n} \subseteq K + L$, and the characteristic polynomial of
$\rKLM(\sigma_i)$ is the exact quotient
$\chi_{\rLMl(\sigma_i)} \big/ \chi_{\rLMl(\sigma_i)|_{K+L}}$ of
characteristic polynomials, the divisibility being guaranteed by the
invariance of $K + L$. Remark \ref{rem:absorb} below is precisely the
scalar-input specialization of this criterion.
\end{remark}

\section{Scalar braid part: classification for semisimple
input}\label{sec:hecke}

In this section, $\kk$ is algebraically closed of characteristic zero.
We consider input representations whose braid part is scalar, and we
assume throughout that $g$ is semisimple. This is a genuine hypothesis
rather than a normalization. A relation $(T - a)(T - b) = 0$ with
$a \neq b$ does force the braid generators to act semisimply on the
quotient (see Section \ref{sec:general}), but semisimplicity of the
output does not imply semisimplicity of the input: the quotient by $K$
may annihilate the nilpotent part of $g$ and return a semisimple
operator. At nonresonant parameters, for $\mu \neq 1$ the hypothesis is
recovered as a consequence
rather than imposed (Remark \ref{rem:absorb}); at $\mu = 1$, however, a
family of non-semisimple inputs does produce Hecke modules, and it lies
outside Theorem \ref{thm:classification}. That family is recorded in
Remark \ref{rem:muone}.

\begin{lemma}\label{lem:scalarrep}
Let $g \in \GL(V)$ and $c \in \kk^{\times}$. Then the assignment
\[
\rho_{g,c}(x_j) := g \quad (1 \le j \le n), \qquad
\rho_{g,c}(\sigma_i) := c\, I_N \quad (1 \le i \le n-1)
\]
defines a representation
$\rho_{g,c} \colon F_n \rtimes B_n \to \GL(V)$.
Conversely, if $\rho$ is a representation of $F_n \rtimes B_n$ with
$\rho(\sigma_i) = c_i I_N$ scalar for all $i$, then $c_1 = \dots = c_{n-1}$
and $\rho(x_1) = \dots = \rho(x_n)$; that is, $\rho = \rho_{g,c}$ for some
$g$ and $c$. In the notation of Remark \ref{rem:twist}, a scalar braid
part forces $C_i = I_N$; the converse fails, since $C_i = I_N$ expresses
only that $s_i$ and $g_i$ commute.
\end{lemma}

\begin{proof}
For the first statement, it suffices to verify the relations
$\sigma_i x_j \sigma_i^{-1} = \theta_{\sigma_i}(x_j)$ under $\rho_{g,c}$.
For $j = i$ we need $c g c^{-1} = g$; for $j = i+1$ we need
$c g c^{-1} = g^{-1} g g$; both hold trivially, as do the remaining cases
and the braid relations among the scalars $c I_N$. For the converse,
\eqref{eq:key} with $s_i = c_i I_N$ gives $g_{i+1} = g_i$ for all $i$, and
the braid relation $\sigma_i \sigma_{i+1}\sigma_i =
\sigma_{i+1}\sigma_i\sigma_{i+1}$ forces $c_i = c_{i+1}$.
\end{proof}

Since all the block entries of $\rDR(x_j)$ and $\rLM(\sigma_i)$ built from
$\rho_{g,c}$ are polynomials in $g$ (with coefficients in
$\kk[\lambda, c]$), the eigenspace decomposition of a semisimple $g$
decomposes the entire construction:

\begin{lemma}\label{lem:decomp}
Let $g$ be semisimple with eigenspace decomposition
$V = \bigoplus_{\nu \in \spec(g)} \Ker(g - \nu I)$, and let
$m_{\nu} = \dim \Ker(g - \nu I)$. Then each subspace $\Ker(g - \nu I)^{\oplus n} \subseteq
V^{\oplus n}$ is invariant under $\rLMl(x_j)$ and $\rLMl(\sigma_i)$ for all
$i, j$; moreover $K = \bigoplus_{\nu} (K \cap \Ker(g - \nu I)^{\oplus n})$ and
$L = \bigoplus_{\nu} (L \cap \Ker(g - \nu I)^{\oplus n})$. Consequently,
\[
\rKLM \;\cong\; \bigoplus_{\nu \in \spec(g)}
\bigl(\rho^{(\nu)}\bigr)^{\oplus m_{\nu}},
\]
where $\rho^{(\nu)}$ denotes the KLM construction applied to the rank-one
input $\rho_{\nu, c} \colon F_n \rtimes B_n \to \GL_1(\kk)$.
\end{lemma}

\begin{proof}
Invariance is immediate since every block entry is a polynomial in $g$ and
$\Ker(g - \nu I)$ is $g$-invariant. For $K$, note that
$K = \Ker(g - I)^{\oplus n}$ (interpreted as $0$ if
$1 \notin \spec(g)$), which is the direct sum of its intersections with
the summands. For $L$, the intersection of kernels of operators preserving
each summand decomposes accordingly. Choosing a basis of each $\Ker(g - \nu I)$
identifies the $\nu$-summand of the construction with $m_{\nu}$ copies of
the rank-one construction with $g = \nu$.
\end{proof}

The rank-one building blocks are computed completely as follows.

\begin{lemma}\label{lem:rankone}
Let $N = 1$, $g = \nu \in \kk^{\times}$, $s_i = c \in \kk^{\times}$, and
$\lambda \in \kk^{\times} \setminus \{1\}$. Write
$S_i := \rLMl(\sigma_i) \in \GL_n(\kk)$. Parts (a) and (b) describe the
removed subspaces $K$ and $L$; part (c) describes the operator $S_i$ on
$\kk^n$ itself, independently of $\lambda$; part (d) describes the
operator induced on the quotient at resonance.
\begin{enumerate}
\item[(a)] If $\nu = 1$, then $K = \kk^n$ and $\rho^{(1)} = 0$.
\item[(b)] If $\nu \neq 1$, then $K = 0$, and $L \neq 0$ if and only if
$\lambda \nu^n = 1$; in that case
$L = \kk\,(\nu^{n-1}, \dots, \nu, 1)^{T}$ is one-dimensional, and
every $S_i$ acts on $L$ by the scalar $c$.
\item[(c)] If $\nu \neq 1$, then
$\det(z - S_i) = (z - c)^{n-1}(z + c\nu)$. If $\nu \neq -1$, the
matrix $S_i$ is semisimple with minimal polynomial
$(T - c)(T + c\nu)$; if $\nu = -1$, then $S_i$ has a nontrivial Jordan
block and minimal polynomial $(T - c)^2$, and
$\operatorname{im}(S_i - cI) = \kk\,(\varepsilon_i - \varepsilon_{i+1})$,
where $\varepsilon_1, \dots, \varepsilon_n$ is the standard basis.
\item[(d)] Suppose $\nu \neq 1$ and $\lambda \nu^n = 1$, and let
$\bar S_i$ denote the operator induced by $S_i$ on $\kk^n / L$. If
$n = 2$, then $\kk^2/L$ is one-dimensional and $\bar S_1 = -c\nu$. If
$n \geq 3$, the minimal polynomial of $\bar S_i$ is $(T - c)(T + c\nu)$
for $\nu \neq -1$, and $(T - c)^2$ for $\nu = -1$ (the latter requiring
$n$ odd and $\lambda = -1$).
\end{enumerate}
\end{lemma}

\begin{proof}
(a) For $\nu = 1$ the off-diagonal entries $\lambda(\nu - 1)$ and
$\nu - 1$ of $\rDR(x_j)$ vanish, and $\Ker(g - 1) = \kk$, so
$K = \kk^n$ and the quotient is zero.

(b) $K = 0$ is clear. By the characterization \eqref{eq:Lchar}, $v \in L$
if and only if $v_k = \nu v_{k+1}$ for $k \le n-1$ and
$v_n = \lambda \nu^n v_n$; the second condition has a nonzero solution
exactly when $\lambda \nu^n = 1$, and then $v$ is proportional to
$(\nu^{n-1}, \dots, \nu, 1)^{T}$. For the action of $S_i$ on such $v$:
the $i$-th entry of $S_i v$ is $c\,\nu\, v_{i+1} = c\, v_i$; the
$(i{+}1)$-st entry is
$c\, v_i + c(1 - \nu)v_{i+1} = c\,\nu v_{i+1} + c v_{i+1} - c\nu v_{i+1}
= c\, v_{i+1}$; all other entries are $c\, v_k$. Hence $S_i v = c v$.

(c) The characteristic polynomial is the case $N = 1$ of Theorem
\ref{thm:charpoly}. The matrix $S_i$ equals $c$ times the identity except
for the $2 \times 2$ principal submatrix
$\bigl(\begin{smallmatrix} 0 & \nu \\ 1 & 1 - \nu \end{smallmatrix}\bigr)$
in rows and columns $i, i+1$, whose eigenvalues are $1$ and $-\nu$. If
$\nu \neq -1$ these are distinct, the $2 \times 2$ block is
diagonalizable, and $S_i$ is semisimple with the two eigenvalues $c$ and
$-c\nu$; since both occur ($n \geq 2$), the minimal polynomial is
$(T - c)(T + c\nu)$. If $\nu = -1$, the block equals
$\bigl(\begin{smallmatrix} 0 & -1 \\ 1 & 2 \end{smallmatrix}\bigr)$, whose
characteristic polynomial is $(z - 1)^2$ while the block is not the
identity; hence it is a nontrivial Jordan block and the minimal polynomial
of $S_i$ is $(T - c)^2$. In that case
$S_i - cI = c\,\bigl(\begin{smallmatrix} -1 & -1 \\ 1 & 1
\end{smallmatrix}\bigr)$ in rows and columns $i, i+1$ and zero
elsewhere, whose image is spanned by $\varepsilon_i - \varepsilon_{i+1}$.

(d) Since $S_i$ acts on $L$ by the scalar $c$ (part (b)), the
characteristic polynomial of $\bar S_i$ is
$(z - c)^{n-2}(z + c\nu)$. For $n = 2$ this equals $z + c\nu$, so
$\bar S_1$ is the scalar $-c\nu$. Let $n \geq 3$. If $\nu \neq -1$,
then $S_i$ is semisimple by (c), hence so is the induced operator
$\bar S_i$, and both eigenvalues occur because $n - 2 \geq 1$; the
minimal polynomial is therefore $(T - c)(T + c\nu)$. If $\nu = -1$
(so that $\lambda(-1)^n = 1$ with $\lambda \neq 1$ forces $n$ odd and
$\lambda = -1$), then $(S_i - cI)^2 = 0$ passes to the quotient, so
the minimal polynomial of $\bar S_i$ divides $(T - c)^2$; it does not
divide $T - c$, because
$\operatorname{im}(S_i - cI) = \kk(\varepsilon_i - \varepsilon_{i+1})$
is not contained in $L = \kk\,(\nu^{n-1}, \dots, \nu, 1)^{T}$, all of
whose coordinates are nonzero while $\varepsilon_i - \varepsilon_{i+1}$
is supported on two coordinates and $n \geq 3$.
\end{proof}

\pagebreak[2]
\begin{theorem}\label{thm:classification}
Let $\kk$ be algebraically closed of characteristic zero, let
$g \in \GL(V)$ be semisimple with $\dim \Ker(g - 1) = k_1$, let
$c \in \kk^{\times}$, and let $\lambda \in \kk^{\times} \setminus \{1\}$.
Write $E := \spec(g) \setminus \{1\}$, $m_\nu := \dim\Ker(g - \nu I)$,
and $\ell := \dim \Ker(\lambda g^n - I_N)$. Consider $\rho = \rho_{g,c}$
as in Lemma \ref{lem:scalarrep} and the resulting KLM representation
$\rKLM$ of $F_n \rtimes B_n$ on $V^{\oplus n}/(K + L)$. Set
\[
p_\nu(T) :=
\begin{cases}
(T - c)(T + c\nu), & \nu \neq -1,\\
(T - c)^2, & \nu = -1.
\end{cases}
\]
Here $\operatorname{lcm}$ denotes the least common multiple of monic
polynomials, normalized to be monic; we use throughout the standard fact
that the minimal polynomial of a direct sum of operators is the least
common multiple of the minimal polynomials of the summands.
Then:
\begin{enumerate}
\item[(i)] $\dim V^{\oplus n}/(K + L) = n(N - k_1) - \ell$.
\item[(ii)] If $E \subseteq \{\mu\}$ with
$\mu \in \kk^{\times} \setminus \{1, -1\}$, then for every
$\lambda \in \kk^{\times}\setminus\{1\}$,
\[
\bigl(\rKLM(\sigma_i) - c\bigr)\bigl(\rKLM(\sigma_i) + c\mu\bigr) = 0,
\qquad i = 1, \dots, n-1,
\]
so that $\sigma_i \mapsto c^{-1}\rKLM(\sigma_i)$ defines a representation
of the Iwahori--Hecke algebra $H_n(\mu)$ of dimension $n\,m_{\mu} - \ell$.
If moreover $m_{\mu} \neq 0$, the operators $\rKLM(\sigma_i)$ are
semisimple with eigenvalue set contained in $\{c, -c\mu\}$; both
eigenvalues occur if and only if $(n-1)\,m_\mu > \ell$, that is, unless
$n = 2$ and $\lambda = \mu^{-2}$.
\item[(iii)] (Complete classification, $n \geq 3$.) Suppose $n \geq 3$
and $E \neq \emptyset$. Then, for every
$\lambda \in \kk^{\times}\setminus\{1\}$, the minimal polynomial of
$\rKLM(\sigma_i)$ equals
$\operatorname{lcm}_{\nu \in E}\, p_\nu$. In particular, it is
quadratic with two distinct roots if and only if $E = \{\mu\}$ with
$\mu \neq -1$, in which case the roots are $c$ and $-c\mu$.
\item[(iv)] (Complete classification, $n = 2$.) Suppose $n = 2$ and
$E \neq \emptyset$, and partition $E = R_\lambda \sqcup U_\lambda$ with
$R_\lambda := \{\nu \in E \,;\, \lambda\nu^2 = 1\}$. Then the minimal
polynomial of $\rKLM(\sigma_1)$ equals
\[
\operatorname{lcm}\Bigl(\{\,T + c\nu \,;\, \nu \in R_\lambda\,\}
\cup \{\,p_\nu \,;\, \nu \in U_\lambda\,\}\Bigr).
\]
It is quadratic with two distinct roots if and only if either
\begin{itemize}
\item[(a)] $E = \{\mu\}$ with $\mu \neq -1$ and $\lambda\mu^2 \neq 1$,
the roots being $c$ and $-c\mu$; or
\item[(b)] $E = \{\nu, -\nu\}$ and $\lambda = \nu^{-2}$, the roots
being $\pm c\nu$. In this case
$\widetilde T_1 := (c\nu)^{-1}\rKLM(\sigma_1)$ satisfies
$(\widetilde T_1 - 1)(\widetilde T_1 + 1) = 0$, so that this additional
normalization exhibits an $H_2(1)$-module.
\end{itemize}
\item[(v)] (Degenerations.) If $E = \{-1\}$, the minimal polynomial is
$(T - c)^2$: the representation factors through the repeated-root
specialization $H_n(-1)$, but through no Hecke relation with two
distinct parameters. If $n = 2$, $E = \{\mu\}$ and $\lambda\mu^2 = 1$,
then $\rKLM(\sigma_1)$ is the scalar $-c\mu$, with linear minimal
polynomial; it still satisfies $(T - c)(T + c\mu) = 0$ and defines an
$H_2(\mu)$-module.
\end{enumerate}
\end{theorem}

\begin{proof}
By Lemma \ref{lem:decomp}, $\rKLM \cong \bigoplus_{\nu}
(\rho^{(\nu)})^{\oplus m_{\nu}}$, and by Lemma \ref{lem:rankone}(a) the
summand $\nu = 1$ vanishes; thus the sum runs over $\nu \in E$, and the
minimal polynomial of $\rKLM(\sigma_i)$ is the least common multiple of
the minimal polynomials of the summands. By Lemma
\ref{lem:rankone}(c),\,(d), the summand at $\nu \in E$ has minimal
polynomial
\begin{equation}\label{eq:mnu}
m_{\nu,\lambda}(T) =
\begin{cases}
p_\nu(T), & \lambda\nu^n \neq 1, \text{ or } \lambda\nu^n = 1 \text{ and }
n \geq 3,\\
T + c\nu, & \lambda\nu^2 = 1 \text{ and } n = 2.
\end{cases}
\end{equation}

(i) Each copy of the summand at $\nu \neq 1$ has dimension
$n - \dim L^{(\nu)}$, where $L^{(\nu)}$ is one-dimensional if
$\lambda \nu^n = 1$ and zero otherwise (Lemma \ref{lem:rankone}(b));
thus each resonant copy loses exactly one dimension. Since $g$ is
semisimple, the number of resonant copies is
$\sum_{\nu:\, \lambda\nu^n = 1} m_{\nu} = \dim\Ker(\lambda g^n - I_N)
= \ell$. Note that $\lambda \neq 1$ excludes $\nu = 1$ from this count.
Summing over $\nu \neq 1$ gives $n(N - k_1) - \ell$.

(ii) By Lemma \ref{lem:rankone}(c), the operator $S_i$ on each
$\mu$-summand satisfies $(S_i - c)(S_i + c\mu) = 0$ already before the
quotient; a polynomial relation satisfied by an operator is inherited by
the induced operator on any quotient by an invariant subspace. Since the
$\nu = 1$ summand contributes zero, the relation holds on the whole
quotient, for every $\lambda$. The braid relations hold because $\rKLM$
is a representation of $F_n \rtimes B_n$, so the rescaled operators
define an $H_n(\mu)$-module by \eqref{eq:heckerel}, of the dimension
computed in (i) with $N - k_1 = m_\mu$. Semisimplicity follows since
$(T - c)(T + c\mu)$ is squarefree. By Remark \ref{rem:scalarcheck}, the
characteristic polynomial on the quotient is
$(z - c)^{(n-1)m_\mu - \ell}(z + c\mu)^{m_\mu}$, so both eigenvalues
occur if and only if $(n-1)m_\mu > \ell$; as $\ell \in \{0, m_\mu\}$
(with $\ell = m_\mu$ exactly at $\lambda = \mu^{-n}$), this fails only
for $n = 2$, $\lambda = \mu^{-2}$.

(iii) For $n \geq 3$, formula \eqref{eq:mnu} gives
$m_{\nu,\lambda} = p_\nu$ for every $\nu \in E$ and every
$\lambda \in \kk^{\times} \setminus \{1\}$, whence the displayed least common multiple. If
$E = \{\mu\}$ with $\mu \neq -1$, the lcm is $(T - c)(T + c\mu)$,
quadratic and squarefree. Conversely, suppose the lcm is quadratic with
two distinct roots. If $-1 \in E$, the factor $(T - c)^2$ makes the lcm
non-squarefree; hence $-1 \notin E$. The roots of the lcm are then
$\{c\} \cup \{-c\nu \,;\, \nu \in E\}$, of cardinality $1 + |E|$ (the
values $-c\nu$ are pairwise distinct and differ from $c$ since
$\nu \neq -1$), so $|E| = 1$.

(iv) The displayed lcm is again immediate from \eqref{eq:mnu}. Suppose
it is quadratic with two distinct roots. If $R_\lambda = \emptyset$,
the argument of (iii) applies verbatim and yields case (a); the
condition $\lambda\mu^2 \neq 1$ is $R_\lambda = \emptyset$. If
$R_\lambda \neq \emptyset$ and $U_\lambda \neq \emptyset$, pick
$\nu \in R_\lambda$ and $\mu' \in U_\lambda$; the lcm is divisible by
$(T - c)(T + c\mu')$ and by $T + c\nu$, and $T + c\nu$ divides
$(T - c)(T + c\mu')$ only if $-c\nu \in \{c, -c\mu'\}$, i.e.
$\nu = -1$ (excluded, since $\lambda\nu^2 = 1$ would force
$\lambda = 1$) or $\nu = \mu'$ (excluded, as $R_\lambda$ and
$U_\lambda$ are disjoint); hence the lcm has degree at least three, a
contradiction. There remains $U_\lambda = \emptyset$, i.e.
$E = R_\lambda \subseteq \{\nu \,;\, \lambda\nu^2 = 1\}$, a set with
exactly two elements $\pm\nu_0$. If $|E| = 1$ the lcm is linear, not
quadratic; hence $E = \{\nu_0, -\nu_0\}$ and the lcm is
$(T + c\nu_0)(T - c\nu_0)$, with the two distinct roots $\mp c\nu_0$.
This is case (b), and $\lambda = \nu_0^{-2}$; note that
$\nu_0 \notin \{1, -1\}$ automatically, since $\lambda \neq 1$. The
normalization statement is immediate.

(v) If $E = \{-1\}$, the lcm is $(T - c)^2$ by \eqref{eq:mnu}
(for $n = 2$ resonance at $\nu = -1$ would force $\lambda = 1$), and
after rescaling this is the defining relation of $H_n(-1)$; a relation
$(T - a)(T - b) = 0$ with $a \neq b$ has squarefree minimal polynomial
and is therefore impossible here. The final statement is
\eqref{eq:mnu} for $n = 2$, $\lambda\mu^2 = 1$: the single summand acts
by the scalar $-c\mu$, which is annihilated by the factor $T + c\mu$ of
$(T - c)(T + c\mu)$.
\end{proof}

\begin{remark}[the polynomial form of the criterion]\label{rem:absorb}
For $\mu \neq 1$, the semisimplicity hypothesis on $g$ in Theorem
\ref{thm:classification} can be absorbed into the statement. Assume
$\lambda^{-1} \notin \spec(g^n)$, so that $K + L = K =
\Ker(g - I)^{\oplus n}$ and the quotient is identified with
$\overline{V}^{\oplus n}$, $\overline{V} := V/\Ker(g - I)$, on which the
blocks act through the map $\bar g$ induced by $g$ on $\overline{V}$. Writing the local block of
$c^{-1}\rKLM(\sigma_i)$ as
$A_{\bar g} = \bigl(\begin{smallmatrix} 0 & \bar g \\ I & I - \bar g
\end{smallmatrix}\bigr)$, a direct computation gives
\[
(A_{\bar g} - I)(A_{\bar g} + \mu I)
=
\begin{pmatrix}
\bar g - \mu I & -\bar g(\bar g - \mu I) \\
-(\bar g - \mu I) & \bar g(\bar g - \mu I)
\end{pmatrix},
\]
which vanishes if and only if $\bar g = \mu I$, that is, if and only if
\[
(g - I)(g - \mu I) = 0
\]
on $V$. For $\mu \neq 1$ the polynomial $(X - 1)(X - \mu)$ is
squarefree, so this identity forces $g$ to be semisimple with
$\spec(g) \subseteq \{1, \mu\}$; the hypothesis of Theorem
\ref{thm:classification} is thus recovered as a consequence rather than
an assumption. Since a Hecke relation with two distinct parameters
itself forces semisimplicity of the braid generators (Remark
\ref{rem:jordan}), we retain the semisimple formulation in the theorem.
The remaining value $\mu = 1$ is treated in Remark \ref{rem:muone}.
\end{remark}

\begin{remark}[the case $\mu = 1$: non-semisimple input]\label{rem:muone}
The computation of Remark \ref{rem:absorb} is valid at $\mu = 1$ as
well, where it produces a family not covered by Theorem
\ref{thm:classification}. Keeping the hypothesis
$\lambda^{-1} \notin \spec(g^n)$, so that
$K + L = K = \Ker(g - I)^{\oplus n}$ and the quotient is identified with
$\overline{V}^{\oplus n}$, $\overline{V} = V/\Ker(g - I)$, the displayed
matrix vanishes for $\mu = 1$ if and only if $\bar g = I$, that is, if
and only if
\[
(g - I)^2 = 0
\]
on $V$. Since $(X - 1)^2$ is not squarefree, this does \emph{not} force
$g$ to be semisimple.

Concretely, let $g = I + \Nil$ with $\Nil \neq 0$, $\Nil^2 = 0$, and let
the braid part be $s_i = c\,I$. Then $g^n = I + n\Nil$, so
$\lambda g^n - I = (\lambda - 1) I + \lambda n \Nil$ is invertible for
every $\lambda \neq 1$ and hence $L = 0$; moreover $\Nil$ induces $0$ on
$\overline{V}$, so $\bar g = I$. The operators
$c^{-1}\rKLM(\sigma_i)$ therefore act on $\overline{V}^{\oplus n}$ by
the permutation matrices of the transpositions $(i, i+1)$; since
$\Nil \neq 0$ gives $\overline{V} \neq 0$, they have minimal polynomial
\[
(T - 1)(T + 1)
\]
with two distinct roots, and realize $\dim \overline{V}$ copies of the
permutation representation of the symmetric group $\mathfrak{S}_n$: an
$H_n(1) \cong \kk\mathfrak{S}_n$-module of dimension $n(N - k_1)$,
independently of $\lambda \in \kk^{\times}\setminus\{1\}$.

Two consequences deserve emphasis. First, the two-distinct-root regime
is not invisible to the non-semisimple part of the input, so the
semisimplicity hypothesis of Theorem \ref{thm:classification} restricts
the classification rather than merely normalizing it; note that $g = I$
and $g = I + \Nil$ have the same spectrum $\{1\}$ but yield,
respectively, the zero quotient and the module above, so that for
non-semisimple input the Hecke property is not a function of $\spec(g)$
alone. Second, although $K$ removes the eigenspace
$\Ker(g - I)^{\oplus n}$, the quotient of the generalized
$1$-eigenspace survives: since $\Nil(V) \subseteq \Ker \Nil$, the
induced operator on $\overline{V}$ is $\bar g = I$. A classification
for non-semisimple $g$ beyond the case
$(g - I)^2 = 0$ recorded above is not attempted here.
\end{remark}

\begin{remark}[failure at $-1$ is invisible to the eigenvalue
set]\label{rem:minusone}
Suppose $\spec(g) = \{1, -1, \mu\}$ with $\mu \neq \pm 1$ and
$\lambda^{-1} \notin \spec(g^n)$ (for $n \geq 3$ this last hypothesis
may be dropped, by Lemma \ref{lem:rankone}(d)). Then by Corollary \ref{cor:quotient}
and Lemma \ref{lem:rankone}, the eigenvalue \emph{set} of
$\rKLM(\sigma_i)$ is exactly $\{c, -c\mu\}$: the $(-1)$-summand
contributes $-c\cdot(-1) = c$, which coincides with the eigenvalue coming
from the scalar part. Nevertheless, the minimal polynomial is
$(T - c)^2 (T + c\mu)$, and the quadratic relation fails. Two eigenvalues
in the set-theoretic sense therefore do not suffice for the Hecke
property; semisimplicity of the braid generators must be required in
addition. More precisely, if $\spec(g)\setminus\{1\} = \{-1\}$, the
minimal polynomial is $(T - c)^2$ alone: after rescaling this is
exactly the defining relation of the repeated-root specialization
$H_n(-1)$, so the representation does factor through a Hecke algebra,
but a non-semisimple one, with a single eigenvalue; it lies outside the
two-distinct-parameter regime classified by Theorem
\ref{thm:classification}. If $-1$ occurs in $\spec(g)$ together with
some $\mu \neq \pm 1$, the
minimal polynomial is $(T - c)^2(T + c\mu)$, of degree three. In terms
of Remark \ref{rem:jordan}, the eigenvalue $-1$ is exactly an overlap
point of the two spectra in \eqref{eq:multiset}:
$-c\cdot(-1) = c \in \spec(s_i)$.
\end{remark}

\begin{remark}[$\lambda$-independence and resonance]\label{rem:lambda}
Within the family of part (ii), the quadratic relation is decided by
$g$ alone, uniformly in $\lambda$; the parameter controls only the
dimension, which drops by $\ell = m_{\mu}$ exactly at the single
resonant value $\lambda = \mu^{-n}$ (which lies in
$\kk^{\times}\setminus\{1\}$ exactly when $\mu^{n} \neq 1$). For
$n = 2$ this resonance in addition degenerates the eigenvalue set
(Theorem \ref{thm:classification}(v)), and outside the family it can
create the two-distinct-root property (Theorem
\ref{thm:classification}(iv)); for $n \geq 3$ neither phenomenon
occurs. This resonance set is the specialization to the present family
of the exceptional set
$\Sigma = \{\lambda\, ;\, \det(\lambda g_1 \cdots g_n - I_N) = 0\}$
(denoted $E$ in \cite{Negami25}) that
governs the degeneration of the invariant Hermitian form in
\cite{Negami25}. In particular, when $\kk = \mathbb{C}$ and
$|\lambda| = 1$, and the input preserves a nondegenerate invariant
Hermitian form (not necessarily definite), the results of
\cite[Section 5]{Negami25} apply and equip these Hecke algebra
representations with an invariant nondegenerate Hermitian form, which
may be indefinite, and whose signature is computable by the algorithm
given there. In the rank-one Burau case, the classical unitarity with
respect to a definite form is due to Squier \cite{Squier84}.
\end{remark}

\begin{remark}[the excluded value $\lambda = 1$]\label{rem:lambdaone}
At $\lambda = 1$ the convolution degenerates. By \eqref{eq:Ediff} the
functions $E_1(v), \dots, E_n(v)$ coincide, so the $n$ defining conditions of $L$ are no
longer equivalent to the chain conditions \eqref{eq:Lchar} and collapse
to the single equation $\sum_{k=1}^{n} (g_k - I)\, v_k = 0$. Thus
$L = \Ker \Phi$ for the map
\[
\Phi \colon V^{\oplus n} \longrightarrow
\textstyle\sum_{j} \operatorname{Im}(g_j - I),
\qquad
\Phi(v_1, \dots, v_n) = \sum_{j=1}^{n} (g_j - I)\, v_j ,
\]
and $K \subseteq L$, whence
\[
V^{\oplus n}/(K + L) \;=\; V^{\oplus n}/L \;\cong\;
\textstyle\sum_{j} \operatorname{Im}(g_j - I).
\]
For scalar braid part this reads
$V^{\oplus n}/(K + L) \cong \operatorname{Im}(g - I) \cong V/\Ker(g - I)$.
Under this identification the actions correspond: for $g_j = g$,
$s_i = cI$ one checks directly that
\[
\Phi \circ \rho^{\mathrm{LM}}_{1}(x_i) = g \circ \Phi,
\qquad
\Phi \circ \rho^{\mathrm{LM}}_{1}(\sigma_i) = c\, \Phi,
\]
so the quotient is equivariantly identified with the input
representation restricted to the invariant subspace
$\operatorname{Im}(g - I)$.
Hence $\rho^{\mathrm{KLM}}_{1}$ recovers the input representation
exactly when $g - I$ is invertible --- in particular in the rank-one
Burau case $g = t \neq 1$ of Example \ref{ex:burau}, in accordance with
the identity $mc_1 = \mathrm{id}$ in Katz's theory
\cite{Katz96, DettweilerReiter00} --- whereas for $g = I$ the output is
zero. The exclusion of $\lambda = 1$ therefore removes the trivial
convolution only up to the contribution of the eigenvalue $1$ of $g$,
which at $\lambda = 1$ is not absorbed by $K$ but already contained in
$L$.
\end{remark}

\begin{example}[Burau]\label{ex:burau}
Let $N = 1$, $g = t \in \kk^{\times}\setminus\{1\}$, $c = 1$. Then
$\rLM(\sigma_i)$ is the $n \times n$ matrix which is the identity except
for the block
$\bigl(\begin{smallmatrix} 0 & t \\ 1 & 1 - t\end{smallmatrix}\bigr)$
in rows and columns $i, i+1$: a variant of the unreduced Burau
representation, with eigenvalues $\{1^{(n-1)}, -t\}$ and quadratic
relation $(T - 1)(T + t) = 0$. For generic $\lambda$ (i.e.,
$\lambda \neq t^{-n}$) we have $K = L = 0$ and
$\rKLM(\sigma_i) = \rLM(\sigma_i)$ is $n$-dimensional. At the resonant
value $\lambda = t^{-n}$ --- available within our setting whenever
$t^{n} \neq 1$, since we require $\lambda \neq 1$ --- the subspace $L$
is spanned by
$(t^{n-1}, \dots, t, 1)^{T}$, on which every $\sigma_i$ acts trivially
(Lemma \ref{lem:rankone}(b)); the quotient is the $(n-1)$-dimensional
reduced Burau representation. Theorem \ref{thm:classification}(ii) with
$\mu = t$ exhibits both as $H_n(t)$-modules, recovering the classical
picture \cite{Jones87}. We note that an explicit reduced form
of the Long--Moody construction, valid for arbitrary input, is given in
\cite[Section 6]{BigelowTian08}; the mechanism here is different, the
reduction arising from the quotient by the resonant subspace $L$ at the
special parameter $\lambda = t^{-n}$ rather than from a modified
construction.
\end{example}

\begin{example}[higher-rank Hecke modules]\label{ex:higher}
Let $g$ be semisimple with $\spec(g) = \{1, \mu\}$, $k_1 = \dim \Ker(g - I)$ and
$m_{\mu} = N - k_1 \geq 1$. For generic $\lambda$, Theorem
\ref{thm:classification} produces an $n\,m_{\mu}$-dimensional
representation of $H_n(\mu)$. By Lemma \ref{lem:decomp} it is isomorphic
to $m_{\mu}$ copies of the ($\lambda$-twisted, unreduced) Burau module of
Example \ref{ex:burau}; the interest of the higher-rank picture lies in
the role of the eigenvalue $1$: each direction of $\Ker(g - I)$ contributes the
eigenvalue pair $\{c, -c\}$ to the intermediate representation
$\rLMl(\sigma_i)$, and the spectrum of $\rLMl(\sigma_i)$ restricted to
$K = \Ker(g - I)^{\oplus n}$ is exactly $\{c\}^{\uplus (n-1)k_1} \uplus
\{-c\}^{\uplus k_1}$, so that the quotient removes precisely these pairs.
This is the mechanism referred to after Theorem \ref{thm:B} in the
introduction.
\end{example}

\begin{remark}[Temperley--Lieb quotient]\label{rem:TL}
The Hecke algebra representations of Theorem
\ref{thm:classification}(ii) in fact factor through the Temperley--Lieb
algebra. Recall that for $\delta \in \kk$ the Temperley--Lieb algebra
$TL_n(\delta)$ is the $\kk$-algebra with generators
$e_1, \dots, e_{n-1}$ and relations
\[
e_i^2 = \delta\, e_i, \qquad
e_i e_{i\pm 1} e_i = e_i, \qquad
e_i e_j = e_j e_i \quad (|i - j| \geq 2).
\]
Given $q \in \kk^{\times}$, choose $r \in \kk^{\times}$ with $r^2 = q$
(possible since $\kk$ is algebraically closed in this section; over a
general base field one first passes to $\kk(r)$) and put
$\delta = r + r^{-1}$. For $q \neq -1$ and $n \geq 3$,
\[
TL_n(\delta) \;\cong\; H_n(q) \big/ \langle A_3 \rangle,
\qquad
A_3 := 1 - T_1 - T_2 + T_1 T_2 + T_2 T_1 - T_1 T_2 T_1,
\]
where $A_3$ is the three-strand $q$-antisymmetrizer in the
normalization \eqref{eq:heckerel} \cite{TemperleyLieb71, Jones87}; in the convention $(T_i + 1)(T_i - q) = 0$ of \cite{Jones87}
the corresponding element has all signs positive. ($A_3$ is not to be
confused with the Jones--Wenzl projector $f_3 \in TL_3$ \cite{Wenzl87},
which is nonzero for generic $\delta$.) For $n = 2$ there is no
three-strand relation and $H_2(q) \cong TL_2(\delta)$, so the statement
below is of interest for $n \geq 3$.

In the setting of Theorem \ref{thm:classification}(ii), write
$T_i := c^{-1}\rKLM(\sigma_i)$, so that $(T_i - 1)(T_i + \mu) = 0$. We
claim that
\begin{equation}\label{eq:TLrel}
(1 - T_i)^2 = (1 + \mu)(1 - T_i), \qquad
(1 - T_i)(1 - T_{i\pm 1})(1 - T_i) = \mu\,(1 - T_i),
\end{equation}
while $(1 - T_i)(1 - T_j) = (1 - T_j)(1 - T_i)$ for $|i - j| \geq 2$ is
automatic from the braid relations. In fact, for the present family
the far products vanish outright,
\begin{equation}\label{eq:farzero}
(1 - T_i)(1 - T_j) = 0 \qquad (|i - j| \geq 2),
\end{equation}
a property of the scalar (Burau-type) output which is strictly stronger
than the far commutation required in $TL_n$ and is not part of the
Temperley--Lieb relations. Granting \eqref{eq:TLrel}, with $r^2 = \mu$ and
$\delta = r + r^{-1}$ as above, the operators
$e_i := r^{-1}(1 - T_i)$ satisfy the Temperley--Lieb relations, so that
the $H_n(\mu)$-module of Theorem \ref{thm:classification}(ii) factors
through $TL_n(r + r^{-1})$; over a base field not containing $r$, the
module satisfies \eqref{eq:TLrel} as it stands and becomes a
$TL_n(r + r^{-1})$-module after the base change to $\kk(r)$.

The identities \eqref{eq:TLrel} follow from a rank-one phenomenon. The
first is the quadratic relation of Theorem \ref{thm:classification}(ii).
For the others, by Lemma \ref{lem:decomp} it suffices to treat a single
copy of the rank-one building block of Lemma \ref{lem:rankone}(c) with
$\nu = \mu$, where $T_i$ is the identity matrix except for the block
$\bigl(\begin{smallmatrix} 0 & \mu \\ 1 & 1 - \mu \end{smallmatrix}
\bigr)$ in rows and columns $i, i+1$. There
\[
1 - T_i = u_i\, w_i^{T}, \qquad
u_i := \varepsilon_i - \varepsilon_{i+1}, \qquad
w_i := \varepsilon_i - \mu\, \varepsilon_{i+1},
\]
where $\varepsilon_1, \dots, \varepsilon_n$ denotes the standard basis
of $\kk^n$. Consequently
$(1 - T_i)(1 - T_j) = (w_i^{T} u_j)\, u_i w_j^{T}$, and the scalars
\[
w_i^{T} u_i = 1 + \mu, \qquad
w_i^{T} u_{i+1} = -\mu, \qquad
w_{i+1}^{T} u_i = -1, \qquad
w_i^{T} u_j = 0 \quad (|i - j| \geq 2)
\]
yield \eqref{eq:TLrel} at once. Being polynomial identities in the
matrix entries, the relations \eqref{eq:TLrel} are inherited by the
quotient at the resonant parameter $\lambda = \mu^{-n}$ (the reduced
Burau case of Example \ref{ex:burau}) and remain valid when $\mu$ is a
root of unity; the semisimplicity of $TL_n(\delta)$ then depends on $n$
and on the order of the parameter. Under the hypotheses of the
unitarity results of \cite{Negami25} ($\kk = \mathbb{C}$,
$|\lambda| = 1$, unitary input), the resulting Temperley--Lieb
representation preserves a nondegenerate, possibly indefinite,
Hermitian form; its definiteness requires a separate analysis. Although
the value $\mu = -1$ lies outside the two-distinct-root regime of
Theorem \ref{thm:classification}, the identities \eqref{eq:TLrel}
themselves persist there: with $r^2 = -1$ one has $\delta = 0$, and the
same formulas yield a possibly nonsemisimple $TL_n(0)$-module. In terms of Young diagrams, the statement reflects the
fact that, generically, the module of Theorem
\ref{thm:classification}(ii) is a multiple of
$S^{(n)} \oplus S^{(n-1,1)}$ and involves only diagrams with at most
two rows.
\end{remark}

\section{The general case: the action on
\texorpdfstring{$K + L$}{K+L}}\label{sec:general}

We return to an arbitrary representation
$\rho \colon F_n \rtimes B_n \to \GL(V)$ and describe the KLM quotient
completely at the level of characteristic polynomials. Throughout this
section $\lambda \in \kk^{\times} \setminus \{1\}$; spectra are taken
over an algebraic closure of $\kk$, and in Proposition
\ref{prop:projection} we assume in addition that the minimal
polynomials of $s_i$ and $s_i g_i$ split over $\kk$ (e.g., that $\kk$
is algebraically closed, as in Section \ref{sec:hecke}). We write
$\chi_A(z) := \det(zI - A)$,
\[
W_j := \Ker(g_j - I) \quad (1 \le j \le n),
\qquad
E_{\lambda} := \Ker(\lambda\, g_1 \cdots g_n - I_N),
\qquad
S_i := \rLMl(\sigma_i).
\]

Two preliminary observations. First, the standard generators $\sigma_i$
are mutually conjugate in $B_n$ for $n \geq 3$; hence the operators
$\rKLM(\sigma_i)$ are mutually conjugate, and their spectra, minimal
polynomials and semisimplicity do not depend on $i$. It therefore
suffices to analyze a single $i$. Second, semisimplicity is not an
additional hypothesis but part of the Hecke condition: a relation
$(T - a)(T - b) = 0$ with $a \neq b$ has squarefree minimal polynomial
and forces the braid generators to act semisimply, while a nontrivial
Jordan block produces a repeated factor and rules out any quadratic
relation with two distinct roots.

\begin{lemma}\label{lem:KL}
For $\lambda \neq 1$ we have $K \cap L = 0$. Since $K$ and $L$ are each
$S_i$-invariant (Lemmas \ref{lem:Laction} and \ref{lem:Kaction} below),
$K + L = K \oplus L$ is an $S_i$-invariant decomposition and
$\chi_{S_i|_{K+L}} = \chi_{S_i|_K}\, \chi_{S_i|_L}$.
\end{lemma}

\begin{proof}
If $v \in K \cap L$, the characterization \eqref{eq:Lchar} gives
$v_k = g_{k+1} v_{k+1}$ with $v_{k+1} \in W_{k+1}$, so
$v_k = v_{k+1}$ for all $k$ and all components equal a common vector
$w \in \bigcap_j W_j$. The last condition in \eqref{eq:Lchar} then reads
$w = \lambda (g_1 \cdots g_n) w = \lambda w$, whence $w = 0$
(cf.\ the proof of \cite[Theorem 16]{Negami25}).
\end{proof}

\begin{lemma}[the action on $L$]\label{lem:Laction}
The Artin action \eqref{eq:artin} fixes the product of the generators,
$\theta_{\sigma_i}(x_1 \cdots x_n) = x_1 \cdots x_n$; consequently
$s_i$ commutes with $g_1 \cdots g_n$ and preserves $E_{\lambda}$.
The map $v \mapsto v_n$ is an isomorphism $L \to E_{\lambda}$
intertwining $S_i|_L$ with $s_i|_{E_{\lambda}}$. In particular
\[
\chi_{S_i|_L}(z) = \chi_{s_i|_{E_{\lambda}}}(z),
\]
so the subspace $L$ can only remove eigenvalues belonging to
$\spec(s_i)$.
\end{lemma}

\begin{proof}
Only the factors $x_i x_{i+1}$ change under $\theta_{\sigma_i}$, and
$\theta_{\sigma_i}(x_i x_{i+1}) = x_{i+1} \cdot x_{i+1}^{-1} x_i x_{i+1}
= x_i x_{i+1}$; applying $\rho$ gives
$s_i (g_1 \cdots g_n) s_i^{-1} = g_1 \cdots g_n$.
By \eqref{eq:Lchar}, an element $v \in L$ is reconstructed from its last
component by $v_k = g_{k+1} g_{k+2} \cdots g_n\, v_n$, and the remaining
condition $v_n = \lambda(g_1 \cdots g_n) v_n$ says exactly
$v_n \in E_{\lambda}$; thus $v \mapsto v_n$ is a bijection
$L \to E_{\lambda}$. For the equivariance, $S_i v \in L$ by invariance,
and its last component is $s_i v_n$: for $i \leq n-2$ the $n$-th block
row of $S_i$ is $s_i$ on the diagonal, while for $i = n-1$,
\[
(S_{n-1} v)_n = s_{n-1} v_{n-1} + s_{n-1}(I - g_n) v_n
= s_{n-1} g_n v_n + s_{n-1}(I - g_n) v_n = s_{n-1} v_n. \qedhere
\]
\end{proof}

\begin{lemma}[the action on $K$]\label{lem:Kaction}
The operator $s_i$ preserves $W_k$ for $k \neq i, i+1$ and maps $W_i$
isomorphically onto $W_{i+1}$, while $s_i g_i$ maps $W_{i+1}$ into
$W_i$. With respect to $K = \bigoplus_j W_j$,
\[
S_i|_K \;=\; \bigoplus_{k \neq i, i+1} s_i|_{W_k}
\;\oplus\; P_i,
\qquad
P_i(w_i, w_{i+1}) := \bigl(s_i g_i\, w_{i+1},\; s_i\, w_i\bigr)
\ \ \text{on } W_i \oplus W_{i+1},
\]
and the characteristic polynomial of the twisted swap $P_i$ --- the
\emph{pair factor} --- is
\[
\chi_{P_i}(z) = \det\!\bigl(z^2 I_{W_i} - (s_i g_i s_i)|_{W_i}\bigr).
\]
\end{lemma}

\begin{proof}
For $k \neq i, i+1$ we have $s_i g_k s_i^{-1} = g_k$, so
$s_i W_k = W_k$; the identity \eqref{eq:key} gives
$s_i W_i = \Ker(s_i g_i s_i^{-1} - I) = W_{i+1}$. Combining
\eqref{eq:key} with $s_i g_{i+1} s_i^{-1} = g_{i+1}^{-1} g_i g_{i+1}$
yields $g_i s_i g_i = s_i g_i g_{i+1}$, so that for
$w \in W_{i+1}$, $g_i(s_i g_i w) = s_i g_i g_{i+1} w = s_i g_i w$,
i.e., $s_i g_i\, W_{i+1} \subseteq W_i$. The displayed block form of
$S_i|_K$ is then read off from Definition \ref{def:LM} exactly as in
the proof of invariance of $K$ \cite{HiroeNegami23, Negami25}. For the
characteristic polynomial, write $P_i = \bigl(\begin{smallmatrix}
0 & A \\ B & 0 \end{smallmatrix}\bigr)$ with
$A = (s_i g_i)|_{W_{i+1}} \colon W_{i+1} \to W_i$ and
$B = s_i|_{W_i} \colon W_i \to W_{i+1}$; a Schur complement computation,
valid as a polynomial identity, gives
$\chi_{P_i}(z) = \det(z^2 I_{W_i} - AB)$ with
$AB = (s_i g_i s_i)|_{W_i}$. Note that
$P_i^2 = (s_i g_i s_i)|_{W_i} \oplus (s_i^2 g_i)|_{W_{i+1}}$, the two
summands being conjugate via $s_i$.
\end{proof}

Assembling Theorem \ref{thm:charpoly} with Lemmas
\ref{lem:KL}--\ref{lem:Kaction} gives a closed formula for the braid
generators on the KLM quotient.

\begin{theorem}\label{thm:generalcharpoly}
Let $\lambda \in \kk^{\times} \setminus \{1\}$ and $1 \le i \le n-1$.
Then
\[
\chi_{\rKLM(\sigma_i)}(z)
=
\frac{\chi_{s_i}(z)^{\,n-1}\,\det(zI + s_i g_i)}
{\displaystyle
\prod_{k \neq i, i+1} \chi_{s_i|_{W_k}}(z)\;\cdot\;
\det\!\bigl(z^2 I_{W_i} - (s_i g_i s_i)|_{W_i}\bigr)\;\cdot\;
\chi_{s_i|_{E_{\lambda}}}(z)},
\]
the denominator dividing the numerator. Moreover, the operators
$\rKLM(\sigma_i)$ satisfy the Hecke relation
$(T - a)(T - b) = 0$ if and only if
\[
(S_i - aI)(S_i - bI)\, V^{\oplus n} \subseteq K + L
\]
(Remark \ref{rem:exact}); when $a \neq b$, this containment forces
semisimplicity and eigenvalue set contained in $\{a, b\}$, i.e., that
the support of the displayed quotient of characteristic polynomials
lies in $\{a, b\}$.
\end{theorem}

\begin{proof}
Since $K + L$ is invariant, characteristic polynomials multiply along
$0 \subseteq K + L \subseteq V^{\oplus n}$; by Lemma \ref{lem:KL} the
restriction to $K + L$ contributes
$\chi_{S_i|_K}\chi_{S_i|_L}$, which is computed by Lemmas
\ref{lem:Laction} and \ref{lem:Kaction}, while the total
characteristic polynomial is given by Theorem \ref{thm:charpoly}. The
second statement is Remark \ref{rem:exact}.
\end{proof}

\begin{remark}\label{rem:scalarcheck}
Specializing to the scalar input of Section \ref{sec:hecke}
($s_i = cI$, $g_j = g$ semisimple) recovers Theorem
\ref{thm:classification}: all $W_k$ equal $\Ker(g-I)$ of dimension
$k_1$, the pair factor becomes
$(z^2 - c^2)^{k_1} = (z-c)^{k_1}(z+c)^{k_1}$, cancelling the
contribution $(z + c)^{k_1}$ of the eigenvalue $1$ of $g$ to
$\det(zI + cg)$, and $\chi_{s_i|_{E_\lambda}}(z) = (z - c)^{\ell}$; the
quotient is
$(z-c)^{(n-1)(N-k_1)-\ell}\prod_{\nu \neq 1}(z + c\nu)^{m_\nu}$.
\end{remark}

\begin{remark}[Temperley--Lieb criterion for general input]
\label{rem:TLgeneral}
For arbitrary input, let $a, b \in \kk^{\times}$, put $\mu := -b/a$ and
$X_i := a^{-1} S_i$, and suppose the quadratic containment
$(X_i - 1)(X_i + \mu)\, V^{\oplus n} \subseteq K + L$ of Theorem
\ref{thm:generalcharpoly} holds, so that the normalized quotient
operators satisfy the Hecke relation. Then they satisfy the
Temperley--Lieb relations \eqref{eq:TLrel} if and only if, in addition,
\[
\Bigl[
\bigl(1 - X_i\bigr)\bigl(1 - X_{i+1}\bigr)\bigl(1 - X_i\bigr)
- \mu\,\bigl(1 - X_i\bigr)
\Bigr] V^{\oplus n} \;\subseteq\; K + L
\qquad (1 \le i \le n-2),
\]
together with the mirror containment with $i$ and $i+1$ interchanged.
The far commutation $(1 - X_i)(1 - X_j) = (1 - X_j)(1 - X_i)$ for
$|i - j| \geq 2$ is automatic, the operators being polynomials in the
commuting $\rKLM(\sigma_i)$, $\rKLM(\sigma_j)$; the stronger vanishing
\eqref{eq:farzero} is particular to the scalar family and is not
required. Unlike the single-generator analysis of Theorem
\ref{thm:generalcharpoly}, this condition involves two adjacent braid
generators simultaneously.
\end{remark}

\begin{corollary}[necessary spectral conditions]\label{cor:necessary}
Suppose the eigenvalue set of $\rKLM(\sigma_i)$ is contained in
$\{a, b\}$. Then every eigenvalue $\nu \notin \{a,b\}$ of $S_i$ must be
absorbed with its full multiplicity:
$m_{\nu}\bigl(S_i|_{K+L}\bigr) = m_{\nu}(S_i)$. In particular, for
$\nu \in \spec(s_i) \setminus \{a, b\}$,
\[
(n-1)\, m_{\nu}(s_i) \;\leq\; \dim K + \dim L .
\]
Consequently, if $\dim K + \dim L < n - 1$, then
$\spec(s_i) \subseteq \{a, b\}$; if moreover $s_i$ is semisimple and is
not a scalar --- in particular, by Remark \ref{rem:twist}, whenever the
twist $C_i$ is nontrivial --- then both $a$ and $b$ occur in
$\spec(s_i)$.
\end{corollary}

Finally, we settle semisimplicity at the intermediate level. Recall
from Remark \ref{rem:jordan} that the only obstruction is the coupling
in the triangular form \eqref{eq:triangular}.

\begin{proposition}[projection criterion]\label{prop:projection}
Suppose $s_i$ and $s_i g_i$ are semisimple, and for an operator $A$ let
$\Pi^{A}_{\nu}$ denote its spectral projection onto the eigenvalue
$\nu$. Then $\rLMl(\sigma_i)$ is semisimple if and only if
\[
\Pi^{s_i}_{\nu}\; \Pi^{-s_i g_i}_{\nu} = 0
\qquad \text{for every } \nu \in \spec(s_i) \cap
\bigl(-\spec(s_i g_i)\bigr).
\]
If this holds, then $\rKLM(\sigma_i)$ is semisimple for every
$\lambda \in \kk^{\times} \setminus \{1\}$. Whether the converse can
fail on the quotient --- that is, whether passing to the KLM quotient
can restore semisimplicity by absorbing a Jordan block into $K + L$ ---
is left open (see Section \ref{sec:discussion}).
\end{proposition}

\begin{proof}
By \eqref{eq:triangular}, $\rLMl(\sigma_i)$ is similar to
$s_i^{\oplus(n-2)} \oplus M_i$ with
$M_i = \bigl(\begin{smallmatrix} A & C \\ O & D \end{smallmatrix}\bigr)$,
$A = C = s_i$, $D = -s_i g_i$ (cf.\ \cite{Roth52}). Conjugating by
$P_Z = \bigl(\begin{smallmatrix} I & -Z \\ O & I \end{smallmatrix}\bigr)$
gives
$P_Z^{-1} M_i P_Z = \bigl(\begin{smallmatrix} A & C - AZ + ZD \\
O & D \end{smallmatrix}\bigr)$.
For eigenvalues $\alpha$ of $A$ and $\beta$ of $D$, write
$Z_{\alpha\beta} = \Pi^{A}_{\alpha} Z\, \Pi^{D}_{\beta}$ and
$C_{\alpha\beta} = \Pi^{A}_{\alpha} C\, \Pi^{D}_{\beta}$; the blocks
with $\alpha \neq \beta$ can be annihilated by the choice
$Z_{\alpha\beta} = (\alpha - \beta)^{-1} C_{\alpha\beta}$, leaving, for
each common eigenvalue $\nu$, the block
$\nu I + \bigl(\begin{smallmatrix} O & C_{\nu\nu} \\ O & O
\end{smallmatrix}\bigr)$. This block is semisimple if and only if
$C_{\nu\nu} = 0$, so $M_i$ is semisimple if and only if
$\Pi^{A}_{\nu} C\, \Pi^{D}_{\nu} = 0$ for every common eigenvalue
$\nu$. Since $A = C = s_i$ and $D = -s_i g_i$, one has
$\Pi^{A}_{\nu} C\, \Pi^{D}_{\nu}
= \nu\, \Pi^{s_i}_{\nu}\, \Pi^{-s_i g_i}_{\nu}$, and $\nu \neq 0$
since $s_i$ is invertible; this gives the displayed condition. Finally,
a quotient of a semisimple operator by an invariant subspace is
semisimple.
\end{proof}

\begin{remark}[division of labour]\label{rem:division}
In the formula of Theorem \ref{thm:generalcharpoly}, the numerator and
the pair factor $\det(z^2 I_{W_i} - (s_ig_is_i)|_{W_i})$ are determined
by the pair $(s_i, g_i)$ alone; the factors
$\chi_{s_i|_{W_k}}$ for $k \neq i, i+1$ and
$\chi_{s_i|_{E_{\lambda}}}$ involve the remaining $g_k$ and the full
product $g_1 \cdots g_n$, hence the tuple itself --- equivalently, in
the notation of Remark \ref{rem:twist}, the twists $C_j$. Thus the
\emph{candidate} Hecke parameters are constrained twist-independently
by Corollary \ref{cor:quotient}, while their attainment, and likewise
the signature of the invariant Hermitian form of \cite{Negami25} in the
unitary setting, depend on the additional data encoded by the tuple.
\end{remark}

\section{Discussion}\label{sec:discussion}

\subsection*{Summary: why the KLM quotient produces Hecke
representations}
The mechanism uncovered here can be summarized as follows, for the
semisimple scalar family of Theorem \ref{thm:classification}. The
eigenvalue $1$ of $g$ is one obstruction to a Hecke relation, and the
one that the construction disposes of structurally: it contributes the
eigenvalue pair $\{c, -c\}$ to the intermediate representation (Theorem
\ref{thm:charpoly}). It is neither the only obstruction --- two
distinct eigenvalues $\mu, \nu \notin \{1\}$ of $g$ already force a
minimal polynomial of degree three (Theorem
\ref{thm:classification}(iii)) --- nor one that every interesting input
carries, the rank-one Burau case $g = t \neq 1$ of Example
\ref{ex:burau} having no eigenvalue $1$ at all. The KLM construction
removes this obstruction
structurally rather than by hypothesis: the quotient by the invariant
subspace $K$ absorbs exactly these pairs, while $L$ removes, at
resonant parameters, a further copy of $\spec(s_i)$ (Lemma
\ref{lem:Laction}). Hecke algebra representations thus arise not
because the input avoids the eigenvalue $1$, but because the
construction itself disposes of its contribution --- which is why the
classification of Theorem \ref{thm:classification} is governed by
$\spec(g) \setminus \{1\}$ alone.

\subsection*{Iteration and Hecke-type input}
Theorem \ref{thm:generalcharpoly} and Proposition \ref{prop:projection}
reduce the Hecke question for arbitrary input to concrete linear
algebra: the candidate parameters are constrained twist-independently;
the parameter $\lambda$ can delete only eigenvalues of $s_i$, through
$L$ (Lemma \ref{lem:Laction}), so that for $n \geq 3$ it can never
create the two-distinct-root Hecke property in the scalar family
(Theorem \ref{thm:classification}(iii)), while for $n = 2$ resonance
can (Theorem \ref{thm:classification}(iv)); whether resonance can
absorb a Jordan obstruction on the quotient for general input remains
open. The remaining
obstruction is the semisimplicity coupling at the spectral overlaps
$\spec(s_i) \cap \bigl(-\spec(s_i g_i)\bigr)$, resolved by the
projection criterion at the intermediate level and by the containment
criterion of Remark \ref{rem:exact} on the quotient. By Corollary
\ref{cor:necessary}, in the twist-nontrivial regime $C_i \neq I$ (with
$s_i$ semisimple and $\dim K + \dim L < n - 1$) both Hecke parameters must already
occur in $\spec(s_i)$ (Corollary \ref{cor:necessary}), so the overlap
regime $\spec(s_i) \cap \bigl(-\spec(s_i g_i)\bigr) \neq \emptyset$ is
the typical one: the
input is then itself of Hecke type.

For input of Hecke type, $\spec(s_i) = \{a, b\}$, Corollary
\ref{cor:necessary} shows that, whenever $K + L$ is too small to absorb
entire eigenvalues, the spectral condition reduces to
$\spec(s_i g_i) \subseteq \{-a, -b\}$; the Hecke property itself is
then decided by the containment criterion of Theorem
\ref{thm:generalcharpoly}. An inheritance result of this flavor is
known in a different setting: Bigelow and Tian proved that their
Lawrence-type construction, whose input is a representation of $B_m$
satisfying a quadratic Hecke relation, produces an output satisfying the
same relation \cite[Theorem 5.4]{BigelowTian08}. The problem below asks
for the analogous classification within the KLM construction, starting
from semidirect-product input. Classifying the pairs $(s_i, g_i)$ with
this property, and determining whether the property can be sustained under
iteration of the KLM construction, appears to be a natural next step; we
note that for Burau-type input the list
$\spec(s_i)^{\uplus(n-1)} \uplus (-\spec(s_i g_i))$ of the intermediate
representation generically contains more than two values, and which of
them survive is governed by $K + L$. A classical benchmark for cubic
behaviour is the Lawrence--Krammer--Bigelow representation
\cite{Lawrence90, Krammer02, Bigelow01}, obtained from the
one-dimensional representation $\sigma_1 \mapsto t$ of $B_2$ by a
generalized Long--Moody construction
\cite[Theorem~5.3]{BigelowTian08}; after a rescaling and change of
parameters, Zinno identified it with the irreducible module of the
Birman--Murakami--Wenzl (BMW) algebra indexed by the one-row Young
diagram with $n - 2$ boxes \cite{Zinno01}, so that its standard
generators satisfy the corresponding cubic relation; the BMW algebra
was introduced independently by Birman--Wenzl and Murakami
\cite{BirmanWenzl89, Murakami87}.

\begin{problem}
Classify the representations $\rho \colon F_n \rtimes B_n \to \GL(V)$
with $\spec(s_i) = \{a,b\}$ and $\spec(s_i g_i) \subseteq \{-a,-b\}$ for
which $\rKLM(\sigma_i)$ is semisimple, and determine the composition
factors of the resulting $H_n$-modules. In particular, determine when the
Hecke property is preserved under iteration of the KLM construction.
\end{problem}

\subsection*{Unitarity and definiteness}
Combining Theorem \ref{thm:classification} with the unitarity results of
\cite{Negami25} produces Hecke algebra representations equipped with an
invariant Hermitian form of computable signature (Remark
\ref{rem:lambda}). Determining the parameters $(c, \mu, \lambda)$ for
which this form is definite is the specialization of
\cite[Problem 18]{Negami25} to the present family and connects, through
the unitary Hecke algebra representations, to the question of Birman and
Brendle \cite{BirmanBrendle05} on the origin of unitary braid group
representations, and to applications in topological quantum computation
\cite{NSSFD08, DRW16}.

\subsection*{Other actions}
By Remark \ref{rem:onlykey}, the triangularization \eqref{eq:triangular}
applies to any Long--Moody-type construction whose evaluated local
block has the form used in Section \ref{sec:charpoly} and satisfies
$g_{i+1} s_i = s_i g_i$; determining for which Wada-type actions
\cite{Wada92, Ito13, Soulie19} this occurs may be investigated
elsewhere, and the extension of the construction to virtual and welded
braid groups will be treated in forthcoming work of the author.

\section*{Acknowledgements}
The author would like to express her sincere gratitude to
Davide Dal Martello, Shigeo Koshitani, Yasuhide Numata and
Masahiko Yoshinaga for valuable discussions, and to Remy Adderton for
suggesting the connection with the Temperley--Lieb algebra, which
motivated Remark \ref{rem:TL}.
This work was supported by JST SPRING, Grant Number JPMJSP2109.

\end{document}